\documentclass[12pt]{amsart}
\usepackage[headings]{fullpage}
\usepackage{amssymb,epic,eepic,epsfig,amsbsy,amsmath,amscd,graphicx}
\usepackage[all]{xy}
\numberwithin{equation}{section}
                        \theoremstyle{plain}
\usepackage{mathrsfs}

\newcommand{\psdraw}[2]
         {\begin{array}{c} \hspace{-1.3mm}
         \raisebox{-4pt}{\psfig{figure=#1.eps,width=#2}}
         \hspace{-1.9mm}\end{array}}

\newtheorem{theorem}{Theorem}[section]
\newtheorem{thm}{Theorem}

\newtheorem{pro}{Proposition}
\newtheorem{lemma}[theorem]{Lemma}
\newtheorem{corollary}[theorem]{Corollary}

\theoremstyle{definition}
\newtheorem{remark}[theorem]{Remark}

\newcommand{\lcr}{\raisebox{-5pt}{\mbox{}\hspace{1pt}
                  \epsfig{file=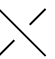}\hspace{1pt}\mbox{}}}
\newcommand{\ift}{\raisebox{-5pt}{\mbox{}\hspace{1pt}
                  \epsfig{file=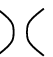}\hspace{1pt}\mbox{}}}
\newcommand{\zer}{\raisebox{-5pt}{\mbox{}\hspace{1pt}
                  \epsfig{file=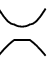}\hspace{1pt}\mbox{}}}

\def\BC{\mathbb C}

\def\BZ{\mathbb Z}
\def\BR{\mathbb R}

\def\CK{\mathcal K}

\def\CR{\mathcal R}
\def\CS{\mathcal S}

\def\fa{\mathfrak a}
\def\fb{\mathfrak b}
\def\fd{\mathfrak d}
\def\fl{\mathfrak l}
\def\fm{\mathfrak m}

\def\fu{\mathfrak u}

\def\la{\langle}
\def\ra{\rangle}

\DeclareMathOperator{\tr}{\mathrm tr}

\def\ve{\varepsilon}
\def\be { \begin{equation} }
\def\ee { \end{equation} }

\begin{document}

\title[The KBSM of two-bridge links]{The Kauffman bracket skein module of two-bridge links}

\author[Thang  T. Q. Le]{Thang  T. Q. Le}
\address{School of Mathematics, 686 Cherry Street,
 Georgia Tech, Atlanta, GA 30332, USA}
\email{letu@math.gatech.edu}

\author[Anh T. Tran]{Anh T. Tran}
\address{School of Mathematics, 686 Cherry Street,
 Georgia Tech, Atlanta, GA 30332, USA}
\email{tran@math.gatech.edu}

\thanks{T.L. was supported in part by National Science Foundation. \\
2010 {\em Mathematics Classification:} Primary 57N10. Secondary 57M25.}

\begin{abstract}
We calculate the Kauffman bracket skein module (KBSM) of the complement of all two-bridge links.
For a two-bridge link, we show that the KBSM of its complement is free over the ring
$\BC[t^{\pm 1}]$ and when reducing $t=-1$, it is isomorphic to the ring of regular functions on the character variety of the link group.
\end{abstract}

\maketitle

\setcounter{section}{-1}

\section{Introduction}

The theory of Kauffman bracket skein module (KBSM) was introduced by Przytycki \cite{Pr} and Turaev \cite{Tu} as a generalization of the Kauffman bracket \cite{Kauffman} in $S^3$  to an arbitrary
$3$-manifold. The KBSM of a knot complement contains a lot, if not all, of information about the colored Jones polynomial. It also contains a lot of information about
classical geometric invariants such as the character variety,  and has been instrumental in the study of the AJ conjecture which relates the colored Jones polynomial and the $A$-polynomial of a knot, see \cite{FGL,Ge,Ga04,Le06,LT}.
The calculation of the KBSM of a knot complement is a difficult task.
At the moment, the KBSM has been calculated only for two-bridge knots \cite{Le06} (with earlier work for twist knots \cite{BL}) and torus knots \cite{Ma} (with earlier work for $(2,2m+1)$-torus knots
\cite{Bu95}).
In this paper, we calculate the KBSM of the complement of all two-bridge links. Applications to the theory of AJ conjecture for links will be discussed in a subsequent work.

\subsection{Skein modules}
\def\cS{\CS}
 A {\em framed link} in an oriented $3$-manifold $Y$ is a disjoint union of embedded circles, each of which is equipped with a non-zero normal vector field. Framed links are considered up to isotopy. In all figures we will draw  framed links,
or part of them, by lines as usual, with the convention that the
framing is blackboard. Let $\mathcal{L}$ be the set of isotopy
classes of framed links in the manifold $Y$, including the empty
link. Consider the free $\BC[t^{\pm 1}]$-module with basis $\mathcal{L}$, and
factor it by the smallest submodule containing all expressions of
the form $\lcr-t\zer-t^{-1}\ift$ and
$\bigcirc+(t^2+t^{-2}) \emptyset$, where the links in each
expression are identical except in a ball in which they look like
depicted. This quotient is denoted by $\CS(Y)$ and is called the
Kauffman bracket skein module, or just skein module, of $Y$.

If $Y_1\subset Y_2$, then the embedding $Y_1\hookrightarrow Y_2$ induces a linear map  $\CS(Y_1) \to \CS(Y_2)$.

For an oriented surface $\Sigma$ we define $\CS(\Sigma)= \CS(Y)$, where $ Y= \Sigma \times [0,1]$, the cylinder over $\Sigma$. The skein module
$\CS(\Sigma)$ has an algebra structure induced by the operation
of gluing one cylinder on top of the other.

\subsection{Main Results}

\label{main}

A two-bridge link is a two-component link $L \subset S^3$ such that
there is a $2$-sphere $S^2\subset S^3$ separating $S^3$ into $2$
balls $B_1$ and $B_2$, and the intersection of $L$ and each ball
is isotopic to $2$ trivial arcs in the ball. The branched double
covering of $S^3$ along a two-bridge link is a lens space
$L(2p,q)$, which is obtained by doing a $2p/q$ surgery on the
unknot. Such a two-bridge link is denoted by $\fb(2p,q)$.
Here $\gcd (q,2p)=1$, and one can always assume that $2p>q \ge 1$. It is known that $\fb(2p',q')$ is isotopic to $ \fb(2p,q)$ if and only if $p'=p$ and $q' \equiv q^{\pm 1} \pmod{2p}$, see \cite{BZ}.

Assume the 3-ball $B_1$ is presented as a vertical cylinder $B_1= D \times [0,1]$, where $D$ is a $2$-dimensional disk, and the two arcs of $L$ inside $B_1$ are two vertical line segments $U \times [0,1]$ and $U' \times [0,1]$,  where $U$ and $U'$ are 2 interior points of $D$. Let $D_{\ast\ast}= D\setminus \{U,U'\}$, then $B_1 \setminus L = D_{\ast\ast} \times [0,1]$. Hence $\cS(B_1 \setminus L)= \cS(D_{\ast\ast})$ is an algebra.
Let $x, x' \subset D_{\ast\ast}$ are respectively small loops around $U,U'$, and $y = \partial D \subset D_{\ast\ast}$ is the boundary of $D$. We consider $x,x',$ and $y$ as elements of the algebra $\cS(B_1 \setminus L)$.
Using the embedding  $(B_1 \setminus L) \subset (S^3\setminus L)$ we will consider $x^a (x')^b\,y^c$ as an element of $\CS(S^3 \setminus L)$.

\begin{figure}[htpb]
$$ \psdraw{}{2.5in} $$
\caption{The loops $x,x'$ and $y$}
\end{figure}


\begin{thm} For the two-bridge link $L=\fb(2p,q)$, the skein module $\CS(S^3 \setminus L)$ is free over $\BC[t^{\pm 1}]$ with basis $\{x^a(x')^by^c \mid 0 \le a,b, \, 0 \le c \le p\}.$
\label{thm}
\end{thm}

\subsection{The universal character ring} 

\label{universal}

Let $\ve (\CS(Y))$ be the quotient of $\CS(Y)$ by the relation $t=-1$. 
 An important result \cite{Bu97,PS} in the theory of skein modules is that
$\varepsilon (\CS(Y))$ has a natural $\BC$-algebra structure and is isomorphic to the universal $SL_2$-character algebra of the fundamental group of $Y$. For a definition of the universal character algebra, see \cite{BH, LM}. The product of 2 links in $\varepsilon
(\CS(Y))$ is their union. Using the skein relation with $t=-1$, it
is easy to see that the product is well-defined, and that the
value of a knot in the skein module depends only on the homotopy
class of the knot in $Y$. The isomorphism between
$\varepsilon(\CS(Y))$ and the universal $SL_2$-character algebra of $\pi_1(Y)$ is given by $K(\rho)= -\text{tr}\, \rho(K)$,
where $K$ is a homotopy class of a knot in $Y$, represented by an
element, also denoted by $K$, of $\pi_1(Y)$, and $\rho:\pi_1(Y) \to
SL_2(\BC)$ is a representation of $\pi_1(Y)$.
The quotient of $\varepsilon (\CS(Y))$
 by its nilradical is canonically isomorphic to
$\BC[\chi(\pi_1(Y))]$, the ring of regular functions on the $SL_2$-character
variety of $\pi_1(Y)$. 

The above fact has been exploited in the work of Frohman, Gelca, and Lofaro \cite{FGL} where they
defined the non-commutative $A$-ideal of a knot, and in our proof of the AJ conjecture \cite{Ga04} for some classes of two-bridge knots and pretzel knots in \cite{Le06, LT}.
In our work on the AJ conjecture, it is important to know whether the universal character algebra $\varepsilon (\CS(Y))$ is reduced, i.e. whether its nilradical is 0.
Although it is difficult to construct a group whose universal character algebra is not reduced (see \cite{LM}),
so far there are a few groups for which the universal character algebra is known to be reduced: free groups \cite{Sikora}, surface groups \cite{Marche_Laurent}, two-bridge knot groups \cite{PS},
torus knot groups \cite{Ma}, and some pretzel knot groups \cite{LT}.

As a consequence of Theorem \ref{thm}, we will show the following.

\begin{pro}
For a two-bridge link $L$, the universal $SL_2$-character algebra  $\varepsilon(\CS(S^3 \setminus L))$ is reduced, and hence $\varepsilon(\CS(S^3 \setminus L))$ is canonically isomorphic to
 the ring of regular functions on the $SL_2$-character
variety of $\pi_1(S^3\setminus L)$. \label{phu}
\end{pro}

\subsection*{Acknowledgements}
The authors would like to thank J. Etnyre, A. Sikora, and the referee for helpful discussions.

\bigskip


\section{Proof of Theorem \ref{thm} and Proposition \ref{phu}}

We change the picture and will present the ball $B_1 \subset \BR^3$ as the closed ball of radius $\sqrt 2$ centered at the origin, i.e. $B_1= \{(x_1,x_2,x_3)\in \BR^3 \mid x_1^2 + x_2 ^2 + x_3 ^2 \le 2\}$.
We suppose that  the two-bridge link $L=\fb(2p,q)$ intersects the interior of
$B_1$ in two straight
intervals $UV$ and $U'V'$ in the $x_1x_2$-plane, where $U=(-1,1,0)$, $U'=(1,1,0),V=(-1,-1,0)$ and $V'=(1,-1,0)$, see Figure \ref{B1}. After an isotopy, we assume that the part
of $L$ outside the interior of $B_1$ are $2$ non-intersecting arcs
$\fu$ and $\fu'$ on the sphere $S=\partial B_1$, where $\fu$ connects $U$ and $V$, and $\fu'$
connects $U'$ and $V'$. If one cuts $S$ along the arc $\fu$, then one obtains a
disk, hence the other arc $\fu'$, is uniquely determined by $\fu$,
up to isotopy.

\begin{figure}[htpb]
$$ \psdraw{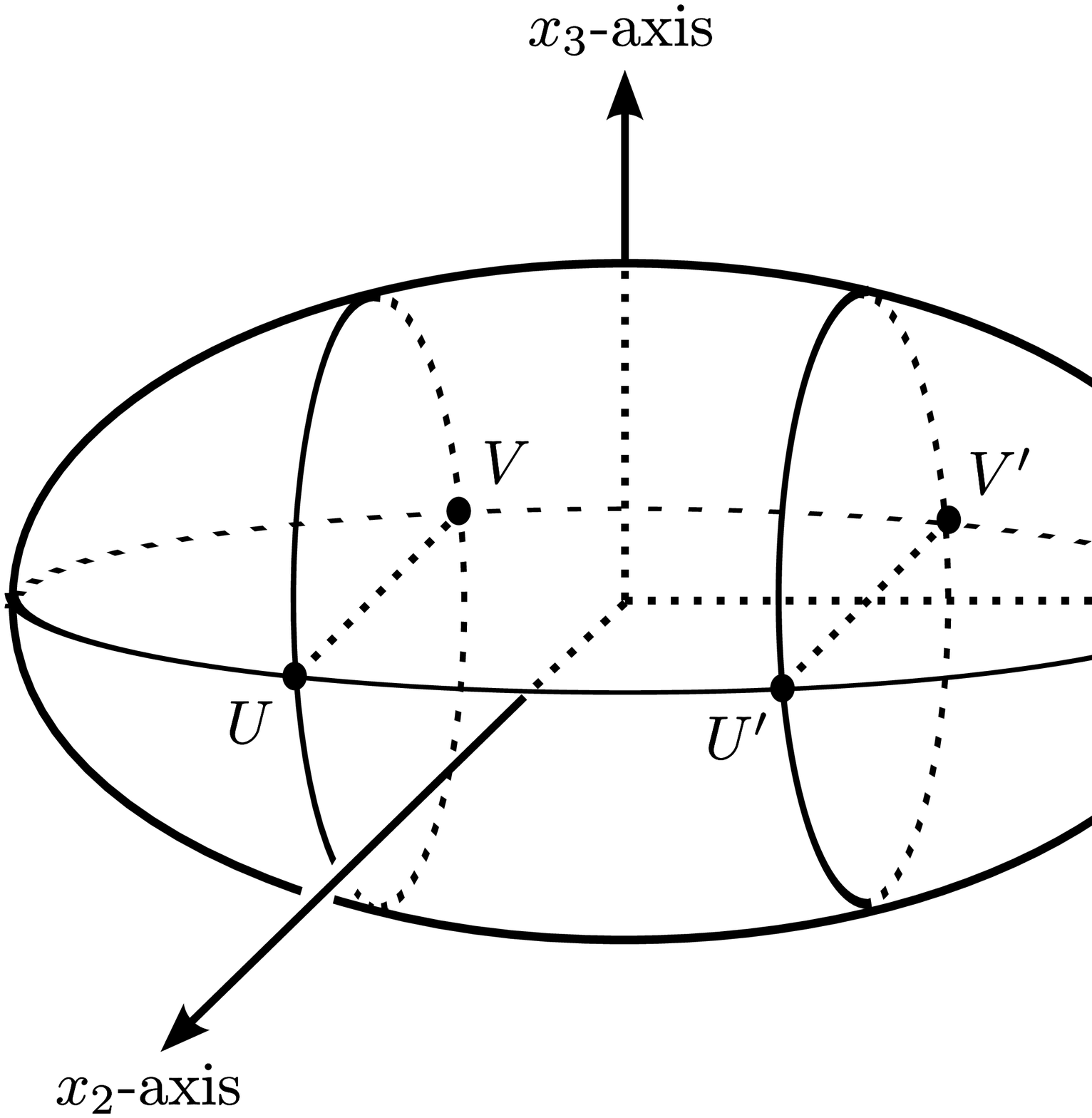}{4.5in} $$
\caption{The ball $B_1$}
\label{B1}
\end{figure}

For a set $Z\subset\BR^3$ let $Z[\alpha,\beta]$ be the part of $Z$ in the
strip $\{ \alpha \le x_1 \le \beta\}$, i.e. $Z[\alpha,\beta] := Z \cap
\{(x_1,x_2,x_3) \mid \alpha\le x_1 \le \beta\}.$

Let $\tilde S$ be the $2$-fold covering of $S$ branched
along the $4$ points $U,U',V,V'$. Note that
$\tilde S$ is a torus, with the following preferred meridian and
longitude. The plane passing through $U,U',V,V'$ (i.e. the $x_1x_2$ plane) intersects
$S[-\sqrt{2},-1]$ in an arc $\fm$ that connects $U$ and $V$. In other words, $\fm$ is the shortest arc on the sphere $S$ connecting $U$ and $V$, see Figure \ref{ml}.
 The total
lift $\tilde \fm$ of $\fm$ is a closed curve on the the torus
$\tilde S$ which will serve as the meridian, see Figure \ref{cover}. Let $\fl$ be the shortest arc on $S$ connecting $U$ and $U'$.  The total lift
$\tilde \fl$ of $\fl$ is a closed curve
serving as the longitude. It is easy to see that $\tilde \fm$
and $\tilde \fl$ form a basis of $H_1(\tilde S,\BZ)$.

\begin{figure}[htpb]
$$ \psdraw{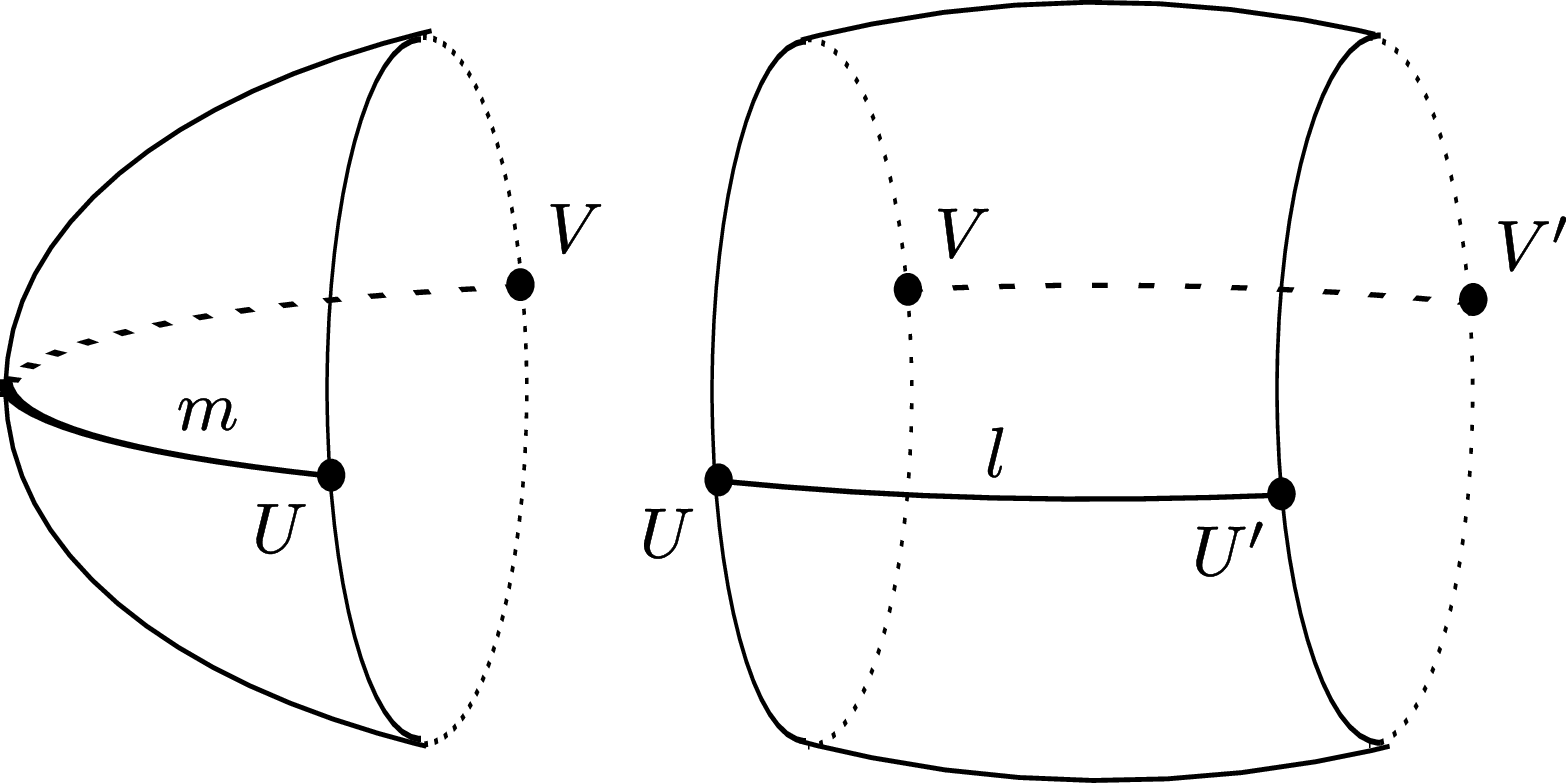}{3.5in} $$
\caption{The curves $\fm$ and $\fl$.}
\label{ml}
\end{figure}

\begin{figure}[htpb]
$$ \psdraw{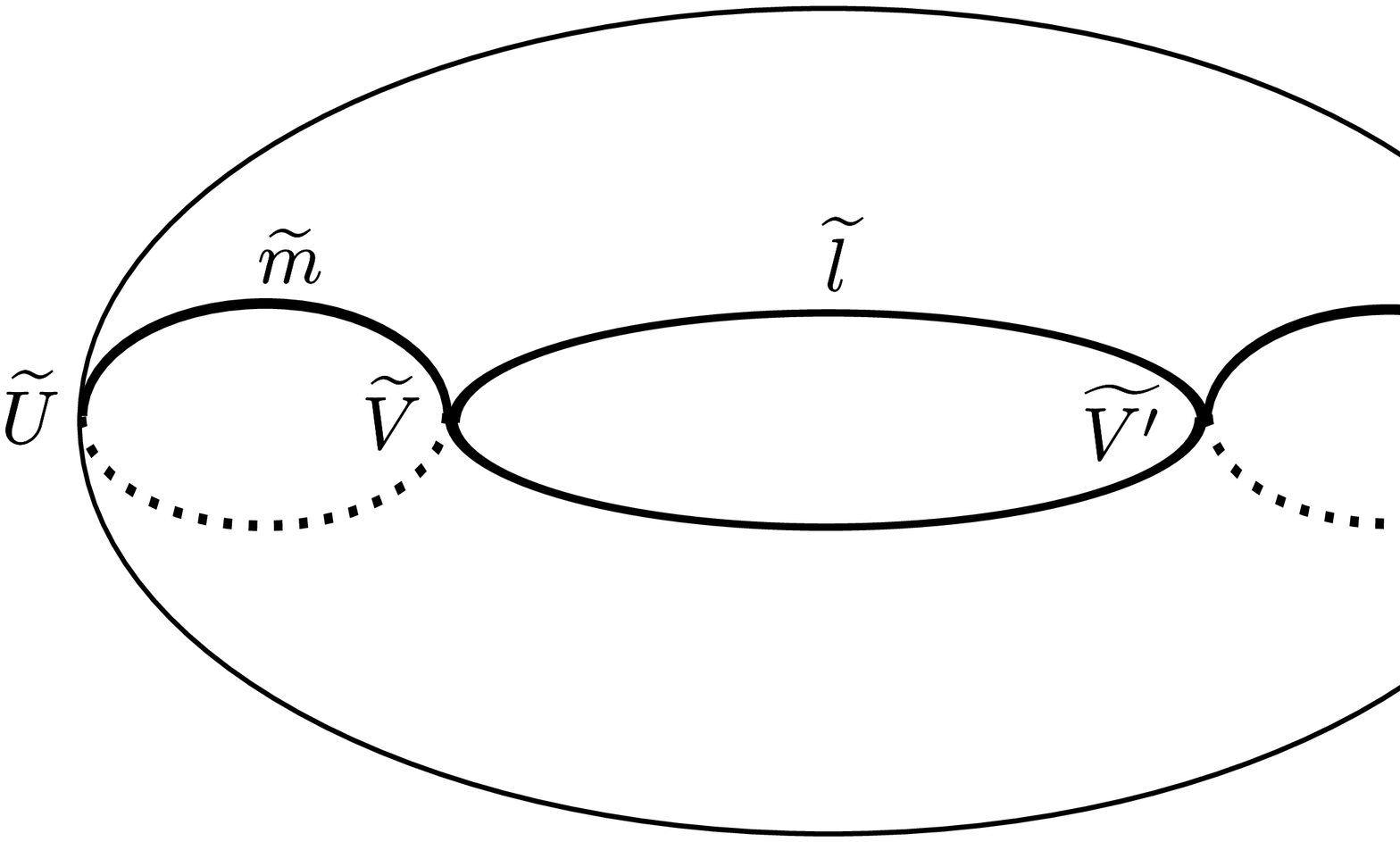}{3.5in} $$
\caption{The total lifts $\tilde{\fm}, \tilde{\fl}$ of $\fm$ and $\fl$ on the torus $\tilde{S}$ respectively.}
\label{cover}
\end{figure}

According to  \cite[Chapter 12]{BZ}, the isotopy class of the pair of arcs $(\fu,\fu')$ in the ball $B_2$ is uniquely determined by the homology class of the total lift $\tilde\fu$ of the curve $\fu$ in $H_1(\tilde S,\BZ).$ Moreover, the homology class of $\tilde\fu$ is equal to
$2p\, \tilde\fm + q' \tilde\fl$ for some $q' \in \BZ$ satisfying the condition $q' \equiv q^{\pm 1} \pmod{2p}.$ We will describe explicitly the arc
$\fu$ in the next subsection.

\subsection{Description of $\fu$} We will present $\fu$ by describing 3 parts of it: the left part $\fu_{l}$, the middle part  $\fu_{m}$, and the right part $\fu_r$, which are respectively the intersection of $\fu$ with $S_l:= S[-\sqrt 2,-1] $,  $S_m:= S[-1,1] $, and  $S_r:=S[1,\sqrt 2] $. For two non-antipodal points $A,B$  on the sphere $S$ let $\gamma(AB)$ be
the shortest geodesic on $S$ connecting $A$ and $B$.

The boundary $C_l:= \partial S_l$  is a circle containing $U$ and $V$. On the circle  $C$
mark $2p$ points $A_0=V,A_1,\dots, A_{2p-1}$ which are: (i) counter-clockwise in that order if viewing from the origin of the coordinate system, and (ii) uniformly distributed on the circle
$C$. 

Then $A_{p}= U$, and for $ 1\le j \le p-1$, the segment $A_{p-j} A_{p+j}$ is parallel to the $x_3$-axis. The shortest geodesic $\gamma(A_{p-j}A_{p+j})$ lies in $S_l$. Let
 $\fu_{0,l}$ be the union of all the disjoint $\gamma(A_{p-j}A_{p+j})$, $1\le j \le p-1$. See Figure \ref{u}.

Let $E_j$ be the midpoint of the arc $A_{j} A_{j+1}$ on the circle $C$ (indices are taken modulo $2p$). In other words, $E_j$ is the image of $A_j$ under the rotation by $2\pi /4p$ about
the $x_1$-axis, counter-clockwise if viewing from the origin. 

Let $E'_j$ be the reflection of $E_j$ through the $x_2x_3$-plane. Note that all the points $E_j'$ are on the circle $C' :=\partial S_r$. The $p$ geodesics $\gamma(E'_{p-j} E'_{p+j-1})$, $j=1,\dots,p$, are disjoint and are in $S_r$. Let $\fu_{0,r}$ be the union of the
$p$ geodesics $\gamma(E'_{p-j} E'_{p+j-1})$, $j=1,\dots,p$.

On $S_m$  let 
$\fu_{0,m}$ be the union of $2p$  geodesics $\gamma(A_j E'_{j+(q-1)/2})$, $j=0,1,\dots,2p-1$ (indices taken modulo $2p)$. Note that the $2p$ components of $\fu_{0,m}$ are obtained from each other 
by rotations by $2j\pi/2p$ about the $x_1$-axis.

\begin{figure}[htpb]
$$ \psdraw{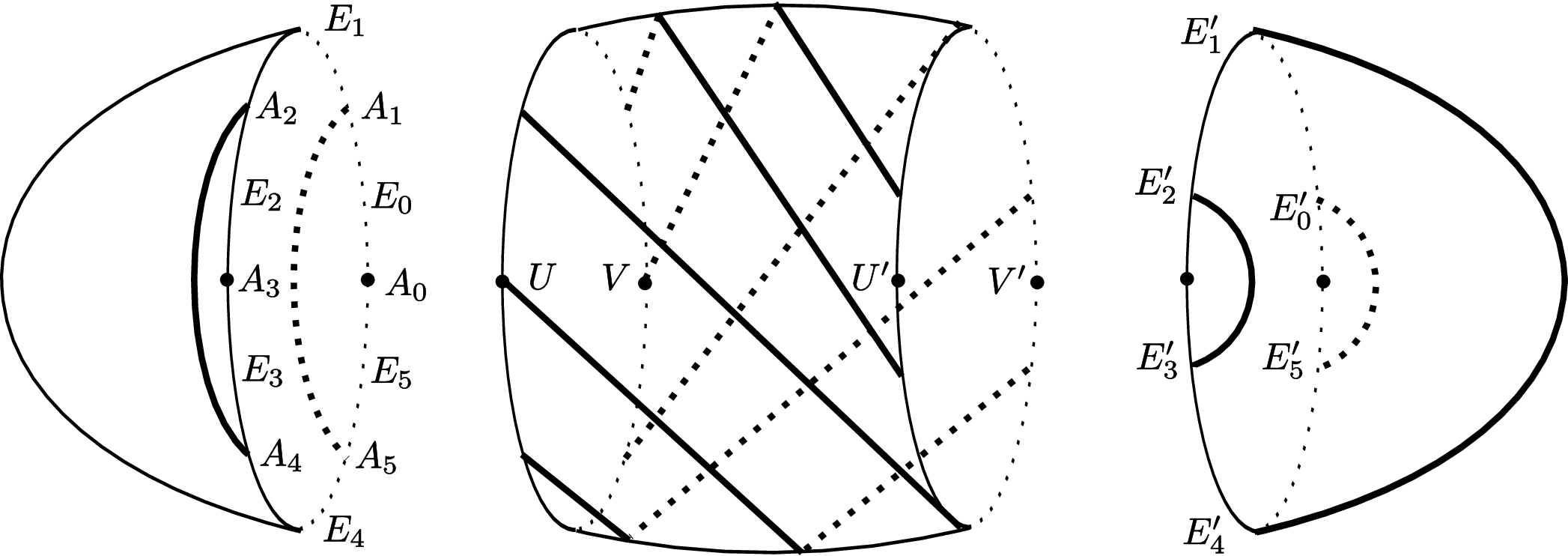}{6in} $$
\caption{$\fu_{0,l}, \fu_{0,m}, \fu_{0,r}$ for $2p=6$ and $q'=5$}
\label{u}
\end{figure}

Let $\fu_0$ be the arc on $S$ obtained by combining $\fu_{0,l}, \fu_{0,m}$ and $\fu_{0,r}$; it connects $U$ and $V$. 
Up to isotopy there is a unique arc $\fu'_0$ on $S$ connecting $V$ and $V'$, and disjoint from $\fu_0$.

\begin{lemma}
The pair $(\fu_0,\fu'_0)$ is isotopic, relative endpoints, to $(\fu,\fu')$ in the ball $B_2.$
\end{lemma}

\begin{proof}
It is easy that the homology class of the total lift of the arc $\fu_0$ in $H_1(\tilde S^2,\BZ)$ is equal to
$2p\, \tilde\fm + q' \tilde\fl$, which is exactly equal to the homology class of the total lift of the arc $\fu$ in $H_1(\tilde S^2,\BZ)$ . According to \cite[Chapter
12]{BZ}, $(\fu_0,\fu'_0)$ is isotopic, relative endpoints, to $(\fu,\fu')$ in the ball $B_2.$
\end{proof}

From now on we identify $(\fu,\fu')$ and $(\fu_0,\fu'_0)$. Without loss of generality we also assume that $q'=q$.

\subsection{The link complement}

Let $\omega$ be the boundary curve of a small normal
neighborhood of the arc $\fu$ in $S=\partial B_1$. Let $X_1:= B_1\setminus (UV \cup U'V')$, which is homeomorphic to the cylinder over a two-punctured disk $D_{**}$ in Subsection \ref{main}. Then the complement $X$ of the link $L$ is obtained from $X_1$ by gluing a $2$-handle to $X_1$ along $\omega$. Up to isotopy, $\omega$ can be described as follows.

On the circle  $C=\partial S_l$
mark $4p$ points $F_0, \, F_1,\dots, F_{4p-1}$ which are: (i) counter-clockwise in that order if viewing from the origin of the coordinate system, and (ii) uniformly distributed on the circle $C$, and such that (iii) $V$ is the midpoint of the arc $F_{4p-1}F_0$ on $C$. We can also say that $F_{2j}$ is the midpoint of the arc $A_{j}E_{j}$, and $F_{2j+1}$ is the midpoint of the arc $E_{j}A_{j+1}$ on $C$, see Figures \ref{afe} and \ref{w1}.

For $ 1\le j \le 2p$, the segment $F_{2p-j} F_{2p+j-1}$ is parallel to the $x_3$-axis. The shortest geodesic $\gamma(F_{2p-j}F_{2p+j-1})$ lies in $S_l$. Let
 $\omega_l$ be the union of all the disjoint $\gamma(F_{2p-j}F_{2p+j-1})$, $1\le j \le 2p$.

Let $F'_j$ be the reflection of $F_j$ through the $x_2x_3$-plane. We can also say that $F'_{2j}$ is the midpoint of the arc $A'_{j}E'_{j}$, and $F'_{2j+1}$ is the midpoint of the arc $E'_{j}A'_{j+1}$ on $C$, where $A'_j$ is the reflection of $A_j$ through the $x_2x_3$-plane. Note that all the points $F'_j$ are on the circle $C' =\partial S_r$. For $ 1\le j \le 2p$, the segment $F'_{2p-j} F'_{2p+j-1}$ is parallel to the $x_3$-axis. The shortest geodesic $\gamma(F'_{2p-j}F'_{2p+j-1})$ lies in $S_r$. Let
 $\omega_r$ be the union of all the disjoint $\gamma(F'_{2p-j}F'_{2p+j-1})$, $1\le j \le 2p$.

Note that we can also say that $\omega_r$ is the reflection of $\omega_l$ through the $x_2x_3$-plane.

On $S_m$, let $\omega_m$  is the union of $4p$  geodesics $\gamma(F_j F'_{q+j})$, $j=0,1,\dots,4p-1$ (indices taken modulo $4p)$. Note that the $4p$ components of $\omega_m$ are obtained  from each other 
by rotations by $2j\pi/4p$ about the $x_1$-axis.

Then, up to isotopy, $\omega$ is obtained by combining $\omega_l, \omega_m$ and $\omega_r$.

\begin{figure}[htpb]
$$ \psdraw{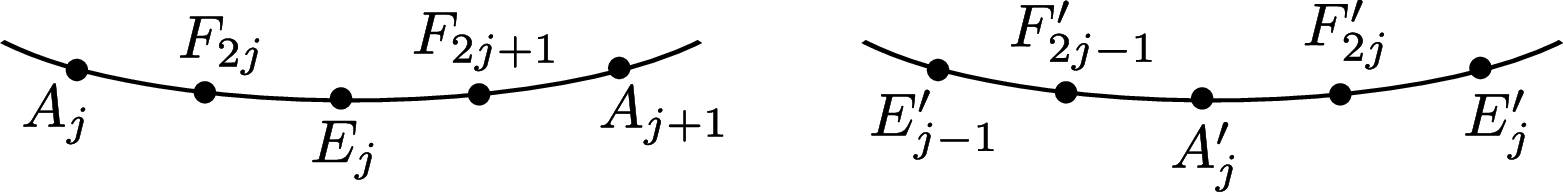}{4in} $$
\caption{The distribution of the points $A_j, F_j, E_j$ and $A'_j, F'_j, E'_j$ on the circles $C$ and $C'$ respectively}
\label{afe}
\end{figure}

\begin{figure}[htpb]
$$ \psdraw{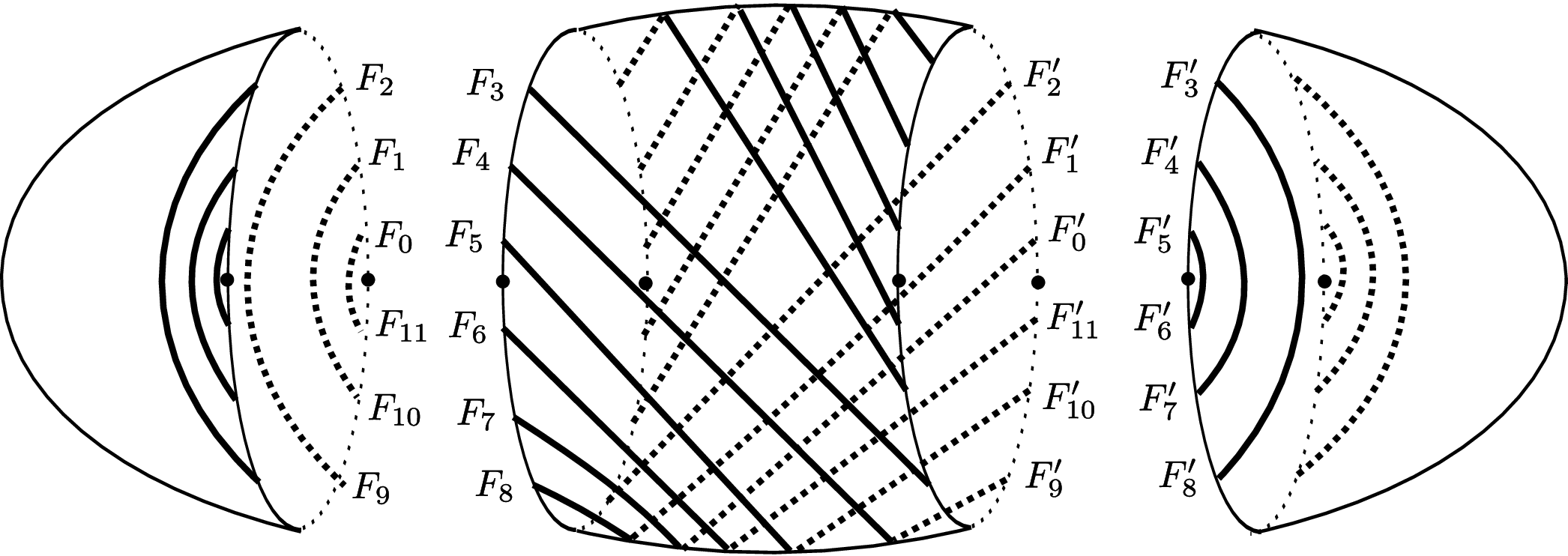}{6in} $$
\caption{$\omega_l, \omega_m, \omega_r$ for $2p=6$ and $q=5$}
\label{w1}
\end{figure}

Let $\psi$ be the rotation by $180^o$ about the $x_1$-axis.  One has
$\psi(B_1)=B_1$. Up to isotopy, we can assume that $\psi(\omega)=\omega$.

Let $P=F_{p+(q+1)/2}, \: P'=F_{3p+(q+1)/2}$ and $Q=F'_{p+1-(q+1)/2}, \: Q'=F'_{3p+1-(q+1)/2}.$ See Figure \ref{sliding}. Note that $\psi(P)=P'$ and $\psi(Q)=Q'$.

\subsection{Relative skein modules}
Let us recall the definition
of the relative skein module $\CS(X_1;P,Q')$ (see \cite{BL, Le06}).
A {\em type 1 tangle} is the disjoint union of a framed link and a
framed arc in $X_1$ such that  the parts of the arc near the two end
points are on the boundary $\partial X_1$, and the framing on these
parts are given by vectors normal to $\partial X_1$. Type 1 tangles
are considered up to isotopy relative the endpoints. Then
$\CS(X_1;P,Q')$ is the $\BC[t^{\pm 1}]$-module generated by type 1 tangles
with endpoints at $P,Q'$ modulo the usual skein relations, like in
the definition of $\CS(X)$. One defines in a similar way the
relative Kauffman bracket skein module $ \CS(\partial X_1;P,Q'):=
\CS(\partial X_1\times [0,1];P,Q')$, where we identify $\partial X_1\times
[0,1]$ with a collar of $\partial X_1$ in $X_1$.

There is a natural bilinear map
$\CS(\partial X_1;P,Q') \otimes \CS(X_1) \to \CS(X_1;P,Q')$, where $\ell
\otimes \ell' \to \ell \star \ell'$, which  is the disjoint union
of $\ell$ and $\ell'$.

The pair $P,Q$ divide $\omega$ into two arcs, the one that is fully
drawn in Figure \ref{sliding} (and that goes around the points $U$ and $V'$ exactly once)
is denoted by $\omega_s$. Similarly, the pair $P',Q'$ divide $\omega$
into two arcs, the one that is fully drawn in Figure \ref{sliding}
(and that goes around the points $U'$ and $V$ exactly once) is denoted by $\omega'_s$.

Let $P_c=F_{3p+1-(q+1)/2}, P'_c=F_{p+1-(q+1)/2}$ and $Q_c=F'_{3p+(q+1)/2}, Q'_c=F'_{p+(q+1)/2}.$
Then $\omega_s$ consists of $3$ parts: the left part is an arc on $S_l$ connecting $P$ and $P_c$,
the middle part is an arc on $S_m$ connecting $P_c$ and $Q_c$, and the right part is an arc on $S_r$ connecting
$Q_c$ and $Q$. Similarly, $\omega'_s$ also consists of $3$ parts:
the left part is an arc on $S_l$ connecting $P'$ and $P'_c$,
the middle part is an arc on $S_m$ connecting $P'_c$ and $Q'_c$,
and the right part is an arc on $S_r$ connecting
$Q'_c$ and $Q'$.

\begin{figure}[htpb]
$$ \psdraw{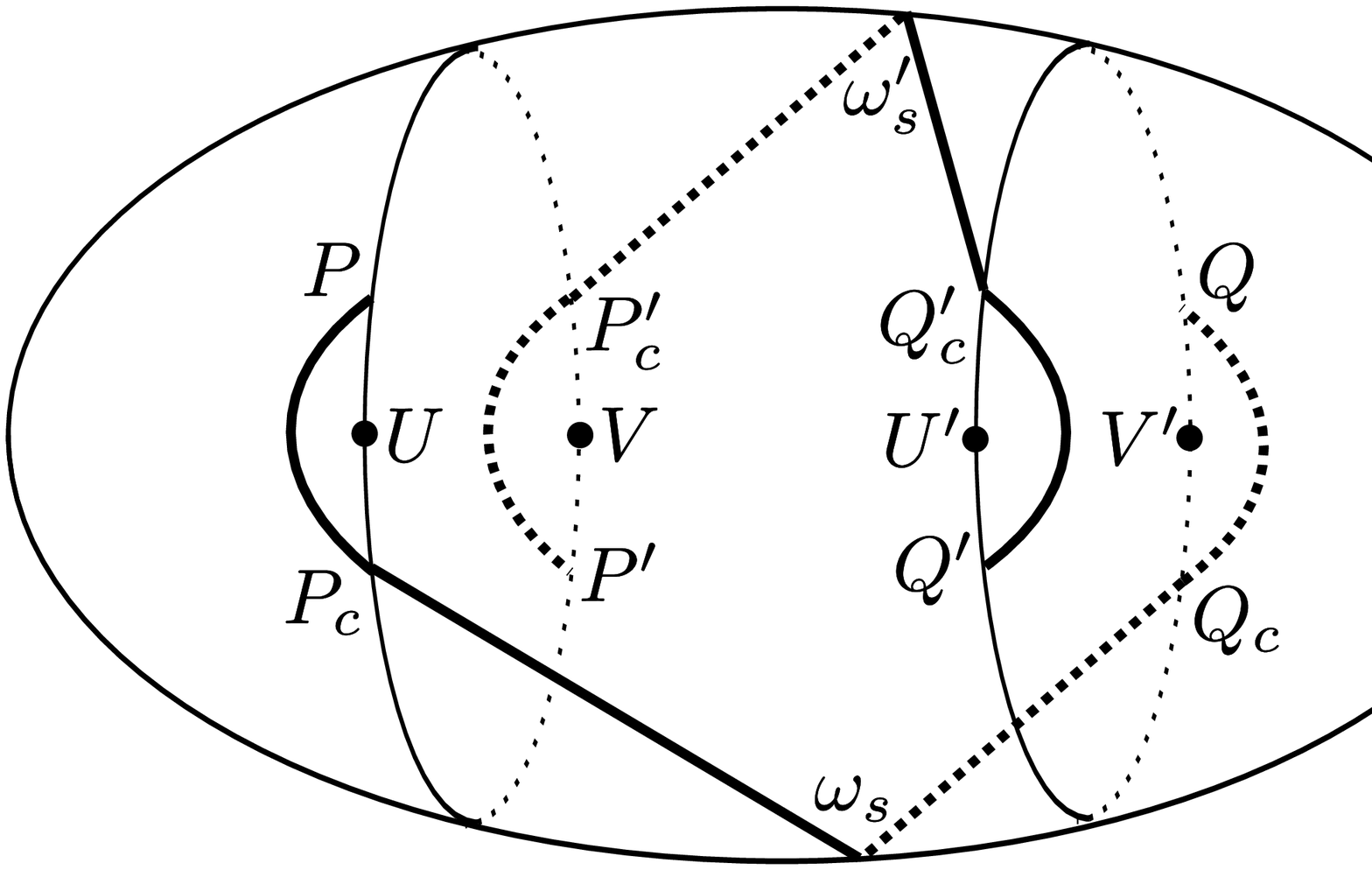}{3.5in} $$
\caption{$\omega_s$ connects $P, Q$, and $\omega'_s$ connects $P', Q'$}
\label{sliding}
\end{figure}

Let $\gamma_{\text{in}}(PQ'), \gamma_{\text{in}}(P'Q)$ be respectively the shortest arcs on the surface $S=\partial B_1$
connecting $P$ and $Q'$, $P'$ and $Q$, whose interiors are slightly pushed inside
the interior of $B_1$ (to avoid intersections with other arcs on $S$) and whose framings are given by vectors
normal to $S$.

Let $\fd_{\text{in}}(PP'), \fd_{\text{in}}(QQ')$ be respectively the straight intervals
connecting $P$ and $P'$, $Q'$ and $Q$, whose interiors are slightly pushed into
the interior of $B_1[-\sqrt{2},-1]$ and the interior of $B_1[1,\sqrt{2}]$ respectively (to avoid intersections with the straight lines $UV$ and $U'V'$ respectively).

Let $\fa_1$ be $\gamma_{\text{in}}(PQ')$; $\fa_2$ be $\omega_s$ followed by
$\fd_{\text{in}}(QQ')$; $\fa_3$ be
$\fd_{\text{in}}(PP')$ followed by $\omega'_s$; and $\fa_4$ be $\omega_s$ followed by $\gamma_{\text{in}}(QP')$ then
followed by $\omega'_s$.

\begin{figure}[htpb]
$$ \psdraw{}{6in} $$
\caption{The projections of $\fa_2, \fa_3$ and $\fa_4$ onto the $x_1x_3$-plane}
\label{sliding1}
\end{figure}

\begin{lemma}
 The relative skein $\CS(X_1; P,Q')$ is equal to the union $\cup_{i=1}^4 (\fa_i \star \CS(X_1)).$
 \label{0}
\end{lemma}

\begin{proof}
Let $\fa'_1=\fa_1$. Let $\fa'_2$ be $\gamma(PP_c)$ folowed by $\gamma_{\text{in}}(P_cQ')$; $\fa'_3$ be $\gamma_{\text{in}}(PQ'_c)$ followed by
$\gamma(Q'_cQ')$; and $\fa'_4$ be $\gamma(PP_c)$ followed by $\gamma_{\text{in}}(P_cQ'_c)$ then
followed by $\gamma(Q'_cQ')$. Here $\gamma_{\text{in}}(P_cQ'),\, \gamma_{\text{in}}(PQ'_c),\, \gamma_{\text{in}}(P_cQ'_c)$ are respectively the shortest arcs on $S$
connecting $P_c$ and $Q'$, $P$ and $Q'_c$, $P_c$ and $Q'_c$, whose interiors are slightly pushed inside
the interior of $B_1$ (to avoid intersections with other arcs on
$S$) and whose framings are given by vectors
normal to $S$.

\begin{figure}[htpb]
$$ \psdraw{}{6in} $$
\caption{The projection of $\fa'_2, \fa'_3$ and $\fa'_4$ onto the $x_1x_3$-plane}
\label{sliding2}
\end{figure}

As in the proof of \cite[Lemma 3.1]{BL}, by using the skein relations one can simplify the arc part of elements in $\CS(X_1;P,Q')$, showing that the arc part is one of the four $\fa'_i, i=1, 2,3,4$. Hence the relative skein $\CS(X_1; P,Q')$ is equal to the union $\cup_{i=1}^4 (\fa'_i \star \CS(X_1)).$

It is easy to see that $\fa_i$ is isotopic to $\fa'_i$ for all $1 \le i \le 3$.
Using the skein relations to resolve all the crossings of $\fa_4$ one can easily show that the set $\fa_4 \star \CS(X_1)$ is equal to $\fa'_4\star \CS(X_1)$, modulo the union $\cup_{i=1}^3 (\fa'_i \star \CS(X_1)).$ The lemma follows.
\end{proof}

\subsection{From $\CS(X_1)$ to $\CS(X)$ through sliding}

Recall that $X$ is obtained from $X_1$ by attaching a $2$-handle
along the curve $\omega$. Note that $\CS(X_1)=\CS(D_{**})$ is isomorphic to the commutative algebra
$\CR[x,x',y]$, see \cite{Pr}. The embedding of $X_1$ into $X$ gives rise to
a linear map from $\CS(X_1)\equiv \CR[x,x',y]$ to $\CS(X)$. It is
known that the map is surjective, and its kernel $\CK$, see
\cite{Pr,BL}, can be described through slides as
follows.

Suppose $\fa$ is a type 1 tangle whose 2 endpoints are on $\omega$ such
that outside a small neighborhood of the 2 endpoints $\fa$ is in the
interior of $X_1$ and in a small neighborhood of the endpoints $\fa$
is on the boundary $S=\partial B_1$. The two end points of $\fa$
divide $\omega$ into 2 arcs $\omega_1$ and $\omega_2$. The loop $\omega$
partitions $S$, which is a sphere, into 2 parts; the
one not containing $U,U'$ is called the {\em outside one}. Let us
isotope $\fa$ (relatively to the endpoints)  to $\fa'$ so that in a
small neighborhood of the endpoints,  $\fa'$ is in the outside part
of $\omega$.

Let $sl(\fa)$ be $\fa' \cdot \omega_1-\fa' \cdot \omega_2$, considered as an element
of the skein module $\CS(X_1)$. Here $\fa' \cdot \omega_1$ is the framed link
obtained by combining  $\fa'$ and $\omega_1$. Note that $sl(\fa)$ is
defined up to a factor $\pm t^{3n}, n\in \BZ$. The exchange $\omega_1
\leftrightarrow \omega_2$ changes the sign, and isotopies in
neighborhoods of the endpoints change the framing, which results in
a factor equal to a power of $(-t^3)$.

It is clear that as framed links in $X$, $\fa' \cdot \omega_1$ is isotopic to
$\fa' \cdot \omega_2$, since one is obtained from the other by sliding over
the 2-handle attached to the curve $\omega$. Hence we always have
$sl(\fa)\in\CK$. It was known that $\CK$ is spanned by all possible
$sl(\fa)$, where $\fa$ can be chosen among all type 1 tangles with
pre-given two endpoints on $\omega$.

From the description of $\CS(X_1; P,Q')$ in Lemma \ref{0} we have

 \begin{lemma} The kernel $\CK$ is equal to the $\BC[t^{\pm 1}]$-span of
$\{sl(\fa_i)\star \CS(X_1), i=1,2,3,4 \}$.
\label{1}
\end{lemma}

\begin{lemma}
 One has
\begin{eqnarray*}
 sl(\fa_1) &=& sl(\gamma_{\emph{in}}(PQ')),\\
sl(\fa_2) &=& sl(\fd_{\emph{in}}(PP')),\\
sl(\fa_3) &=& sl(\fd_{\emph{in}}(QQ')),\\
sl(\fa_4) &=& sl(\gamma_{\emph{in}}(P'Q)).
\end{eqnarray*}
\label{2}
\end{lemma}

\begin{proof}
The first identity is a tautology. The last three follows from trivially a simple isotopy of the links involved.
\end{proof}

\begin{lemma}
For every $\ell \in \CS(X_1),$ one has $\psi(\ell)=\ell.$
\label{symmetry}
\end{lemma}

\begin{proof}
This is because $x,x'$ and $y$ are invariant under the rotation $\psi.$
\end{proof}

\begin{lemma}
 One has
\begin{eqnarray*}
sl(\gamma_{\emph{in}}(PQ')) \star \CS(X_1) &=&  sl(\gamma_{\emph{in}}(P'Q)) \star \CS(X_1),\\
sl(\fd_{\emph{in}}(PP')) \star \CS(X_1) &=& 0,\\
sl(\fd_{\emph{in}}(QQ')) \star \CS(X_1) &=& 0.
\end{eqnarray*}
\label{4}
\end{lemma}

\begin{proof}
Since $\psi(P)=P'$ and $\psi(Q)=Q'$, we have $\psi(sl(\gamma_\text{in}(PQ'))) = sl
(\gamma_\text{in}(P'Q)).$ Hence $sl(\gamma_\text{in}(PQ')) \star \CS(X_1) =  sl(\gamma_\text{in}(P'Q)) \star \CS(X_1)$ by Lemma \ref{symmetry}.

Since both
$\fd_\text{in}(PP')$ and $\omega$ is invariant under $\psi$, we have
$\psi(\fd_\text{in}(PP') \cdot \omega_1(P,P')) = \fd_\text{in}(PP') \cdot\omega_2(P,P'),$ where
$\omega_1(P,P')$ and $\omega_2(P,P')$ are the two arcs of $\omega$ obtained
by dividing $\omega$ using the two points $P,P'$. It implies that
$$sl(\fd_\text{in}(PP')) \star \CS(X_1) = \left( \fd_\text{in}(PP') \cdot \omega_1(P,P') -
\fd_\text{in}(PP') \cdot \omega_2(P,P') \right) \star \CS(X_1)=0.$$
This completes the proof of the lemma.
\end{proof}

\subsection{Proof of Theorem \ref{thm}}
Let $\CR=\BC[t^{\pm 1}]$. We have $\CS(X)=
 \CR[x,x',y]/\CK,$ where $\CK$ is the
 $\CR$-span of $sl(\fa_1)\star\CR[x,x',y]$, by Lemmas \ref{1}, \ref{2} and \ref{4}. Note that there is a
 natural $\CR[x,x']$-module structure on $\CS(X)$: Here $x,x'$ are
 meridians, thus belong to the boundary of $X$. Over $\CR[x,x']$,
 $\CR[x,x',y]$ is spanned by $1,y,y^2,\dots$. Hence $\CK$, as an
 $\CR[x,x']$-module, is spanned by $sl(\fa_1)\star y^k= (\fa_1 \cdot \omega_1-\fa_1 \cdot \omega_2)\star y^k ,
 k=0,1,2,\dots$.

Note that $\fa_1 \cdot \omega_1$ is the closure in the sense of \cite[Section 1.5]{Le06}
 of a braid on $(2p+2)$ strands, while $\fa_1 \cdot \omega_2$ is the
closure  of a braid on $(2p-2)$ strands. Moreover, $(\fa_1 \cdot \omega_1)
\star y^k$ is the closure of a braid on $(2p+2) + 2k$ strands,
while $(\fa_1 \cdot \omega_2) \star y^k$ is the closure of of a braid on
$(2p-2) + 2k$ strands. Lemma $1.1$ in \cite{Le06} then shows that
$(\fa_1 \cdot \omega_1-\fa_1 \cdot \omega_2)\star y^k$, as an element of $\CR[x,x',y]$,
has $y$-degree $(p+1)+ k$, with highest coefficient invertible
and of the form a power of $t$. Hence when we factor out
$\CR[x,x',y]$ by $\CK$, we get a free $\CR[x,x']$-module with
representatives $y^l, l=0,1,2,\dots, p$ as a basis.

\subsection{Proof of Proposition \ref{phu}}
From Theorem 1 it follows that $\ve(\CS(X))$ is the quotient of the ring $\BC[\bar{x},\bar{x'},\bar{y}]$ by the ideal $I$ generated by $\varepsilon(sl(\fa_1))\star\BC[\bar{x},\bar{x'},\bar{y}]$, where $\bar{z}$ denotes the negative of the trace of the loop $z$. 

Note that $\ve(\CS(X))$ has a natural $\BC$-algebra structure and $\star$ is just the multiplication of this algebra. It implies that $I$ can be generated by only one element which is $\varepsilon(sl(\fa_1))$. Hence $\ve(\CS(X))=\BC[\bar{x},\bar{x'},\bar{y}]/(\varphi)$, where $\varphi=\varepsilon(sl(\fa_1)) \in \BC[\bar{x},\bar{x'},\bar{y}].$ Note that $\varphi$ is a polynomial of $\bar{y}$-degree $p+1$ with leading coefficient $\pm 1$. 

We claim that $\varphi$ has no repeated factors. Since $\varphi$ is a polynomial of $\bar{y}$-degree $p+1$ with leading coefficient $\pm 1$, it suffices to show that $\varphi(0,0,\bar{y})$ has no repeated factors.

\begin{lemma}
 One has $$\varphi(0,0,\bar{y})=\pm (\bar{y}^2-4)S_{p-1}(\bar{y}),$$ where $S_n(\bar{y})$ are the Chebyshev polynomials defined by $S_0(\bar{y})=1, S_1(\bar{y})=\bar{y}$ and $S_{n+1}(\bar{y})=\bar{y}S_n(\bar{y})-S_{n-1}(\bar{y})$ for all integer $n$.
\label{00}
\end{lemma}

\begin{proof}
By \cite{BZ}, the fundamental group of the two-bridge link $L=\fb(2p,q)$ is
$$\pi_1(L) = \la \tilde{x}, \tilde{x'} \mid \tilde{x}w=w\tilde{x} \ra,$$
where $w=(\tilde{x'})^{\varepsilon_1}(\tilde{x})^{\varepsilon_2} \cdots (\tilde{x})^{\varepsilon_{2p-2}}(\tilde{x'})^{\varepsilon_{2p-1}}$ and $\varepsilon_k=(-1)^{\lfloor \frac{kq}{2p} \rfloor}.$ Here $\tilde{x},\tilde{x'}$ are meridians of the link $L$, and are conjugate to $x,x'$ respectively.

The character variety of the free group in 2 letters $\tilde{x}$ and $\tilde{x'}$ is isomorphic to $\BC^3$, by the Fricke-Klein-Vogt theorem. For every word $z$, the trace of $z$ is a polynomial in 3 variables $\tr \tilde{x}=-\bar{x}, \tr \tilde{x'}=-\bar{x'}$ and $\tr (\tilde{x}\tilde{x'})=-\bar{y}.$

Note that the traces of the words $(\tilde{x})^{-1}w\tilde{x}(\tilde{x'})^{-1}$ and $w(\tilde{x'})^{-1}$ are equal. Hence $$\eta=\tr \left( (\tilde{x})^{-1}w\tilde{x}(\tilde{x'})^{-1} - w(\tilde{x'})^{-1} \right)$$ is divisible by $\varphi$ in $\BC[\bar{x},\bar{x'},\bar{y}].$

Suppose from now on $\bar{x}=\bar{x'} =0$. We have $(\tilde{x})^{-1}+\tilde{x}=\tr \tilde{x} = -\bar{x}=0$, i.e. $(\tilde{x})^{-1}=-\tilde{x}$, by the Cayley-Hamilton theorem applying for matrices in $SL_2(\BC)$. Here we identify $\tilde{x}$ with its representation matrix in $SL_2(\BC).$
Similarly, $(\tilde{x'})^{-1}=-\tilde{x'}.$

Let $k$ be the the number of times the power $-1$ appears in the word $w(\tilde{x'})^{-1}=(\tilde{x'})^{\varepsilon_1}(\tilde{x})^{\varepsilon_2} \cdots (\tilde{x})^{\varepsilon_{2p-2}}$. Then it is easy to see that the number of times the power $-1$ appears in the word $(\tilde{x})^{-1}w\tilde{x}(\tilde{x'})^{-1}=(\tilde{x})^{-1}(\tilde{x'})^{\varepsilon_1}(\tilde{x})^{\varepsilon_2} \cdots (\tilde{x})^{\varepsilon_{2p-2}}(\tilde{x'})^{\varepsilon_{2p-1}}\tilde{x}(\tilde{x'})^{-1}$ is $k+2.$ If we replace $(\tilde{x})^{-1}$ and $(\tilde{x'})^{-1}$ in $w(\tilde{x'})^{-1}$ by $\tilde{x}$ and $\tilde{x'}$ respectively then we pick up the sign $(-1)^k$, i.e. we have $w(\tilde{x'})^{-1}=(-1)^k(\tilde{x'}\tilde{x})^{p-1}.$ Similarly, $(\tilde{x})^{-1}w\tilde{x}(\tilde{x'})^{-1}=(-1)^{k+2}(\tilde{x}\tilde{x'})^{p+1}$. It implies that $\eta(0,0,\bar{y})=(-1)^k \tr \left( (\tilde{x}\tilde{x'})^{p+1}- (\tilde{x'}\tilde{x})^{p-1} \right)$.

Let $\delta_n=\tr \left( (\tilde{x}\tilde{x'})^{n+1}- (\tilde{x'}\tilde{x})^{n-1} \right).$ By the Cayley-Hamiton, $\tilde{x}\tilde{x'}+(\tilde{x}\tilde{x'})^{-1}=\tr (\tilde{x} \tilde{x'})= -\bar{y}$. This implies that $\delta_{n+1}=-\bar{y}\delta_n-\delta_{n-1}.$ It is easy to check that $\delta_1=\bar{y}^2-4,\:\delta_2=-(\bar{y}^2-4)\bar{y}.$ Hence
 $\delta_n=(-1)^{n-1} (\bar{y}^2-4)S_{n-1}(\bar{y})$ where $S_n(\bar{y})$ are the Chebyshev polynomials defined by $S_0(\bar{y})=1, S_1(\bar{y})=\bar{y}$ and $S_{n+1}(\bar{y})=\bar{y}S_n(\bar{y})-S_{n-1}(\bar{y})$ for all integer $n$. 

We have $\eta(0,0,\bar{y})=(-1)^k\delta_p=(-1)^{k+p-1}(\bar{y}^2-4)S_{p-1}(\bar{y})$, which is a polynomial of degree $p+1$ in $\bar{y}$ with leading coefficient $(-1)^{k+p-1}.$ Since $\eta$ is divisible by $\varphi$, and $\varphi$ is also a polynomial of $\bar{y}$-degree $p+1$ with leading coefficient $\pm 1$, we must have $\varphi(0,0,\bar{y})=\pm (\bar{y}^2-4)S_{p-1}(\bar{y})$ as desired.
\end{proof}

It is known that $S_{p-1}(\bar{y})=\prod_{j=1}^{p-1} (\bar{y}-2\cos \frac{\pi j}{p})$ and hence $(\bar{y}^2-4)S_{p-1}(\bar{y})$ has no repeated factors. By Lemma \ref{00}, it follows that $\varphi$ has no repeated factors either. Hence the nil-radical of $\ve(\CS(X))$ is zero, which means that $\ve(\CS(X))$ is exactly equal to $\BC[\chi(\pi_1(X))]$. This completes the proof of Proposition \ref{phu}.

\begin{corollary}
The character ring of the two-bridge link $\fb(2p,q)$ is the quotient of the ring $\BC[\bar{x}, \bar{x'},\bar{y}]$ by the ideal generated by the polynomial $\eta=\tr ((\tilde{x})^{-1}w\tilde{x}(\tilde{x'})^{-1}) - \tr (w(\tilde{x'})^{-1})$, where $\bar{x}, \, \bar{x'}, \, \bar{y}, \, \tilde{x}, \, \tilde{x'}$ and $w$ are defined as in the proof of Lemma \ref{00}.
\label{character}
\end{corollary}

\begin{proof}
We still use the notations in the proof of Lemma \ref{00}. 

Since $w=(\tilde{x'})^{\varepsilon_1}(\tilde{x})^{\varepsilon_2} \cdots (\tilde{x})^{\varepsilon_{2p-2}}(\tilde{x'})^{\varepsilon_{2p-1}}$ and $\varepsilon_k=(-1)^{\lfloor \frac{kq}{2p} \rfloor}=\pm 1$, it is easy to show that the traces of the words $(\tilde{x})^{-1}w\tilde{x}(\tilde{x'})^{-1}$ and $w(\tilde{x'})^{-1}$ have $\bar{y}$-degrees equal to $p+1$ and $p-1$ respectively, with leading coefficients $\pm 1$. It implies that the polynomial $\eta$ has $\bar{y}$-degree $p+1$ with leading coefficient $\pm 1$. Since $\eta$ is divisible by $\varphi,$ we must have $\eta=\pm \varphi.$ Hence, by Proposition \ref{phu}, the character ring of $\fb(2p,q)$ is equal to the quotient of the ring $\BC[\bar{x}, \bar{x'},\bar{y}]$ by the ideal generated by the polynomial $\eta=\tr ( (\tilde{x})^{-1}w\tilde{x}(\tilde{x'})^{-1}) - \tr(w(\tilde{x'})^{-1})$.
\end{proof}

\begin{remark}
 Corollary \ref{character} was already obtained in \cite{Ri} although it was not completely written in form of traces. The proof we present here essentially follows directly from Theorem \ref{thm}. 
 
 One can easily show that the characters of abelian representations (into $SL_2(\BC)$) of the two-bridge link $\fb(2p,q)$ is determined by the polynomial $$\eta_{\emph{ab}}=\tr(\tilde{x}\tilde{x'}(\tilde{x})^{-1}(\tilde{x'})^{-1})=\bar{y}^2+\bar{x}^2+\bar{x'}^2+\bar{y}\bar{x}\bar{x'}-4.$$
Hence, by Corollary \ref{character}, the characters of non-abelian representations of $\fb(2p,q)$ is determined by the polynomial $\eta_{\emph{nab}}=\eta/\eta_{\emph{ab}}$. The polynomial $\eta_{\emph{nab}}$ has $\bar{y}$-degree $p-1$ with leading coefficient $\pm 1.$ After a suitable change of variables, it is exactly the polynomial $\Phi_{\pi L}$ in \cite[Lemma 2]{Ri}, up to $\pm 1$.
\end{remark}

\end{document}